\documentclass[11pt]{article}
\usepackage[utf8]{inputenc}

\usepackage{amsfonts, bbold, color}
\usepackage{amsmath, amsthm}
\usepackage{amssymb}
\usepackage{indentfirst}
\usepackage{hyperref}
\usepackage[matrix,arrow,curve]{xy}

\usepackage[english]{babel}

\renewcommand{\le}{\leqslant}
\renewcommand{\ge}{\geqslant}

\newcommand{\R}{\mathbb R}
\renewcommand{\P}{\mathbb P}

\newcommand{\E}{\mathbb E}

\newcommand{\Z}{\mathbb Z}
\newcommand{\N}{\mathbb N}
\newcommand{\F}{\mathcal F}
\newcommand{\J}{\mathcal J}

\newcommand{\A}{\mathcal A}
\newcommand{\G}{\mathcal G}
\newcommand{\C}{\mathcal C}
\newcommand{\B}{\mathcal B}
\newcommand{\K}{\mathcal K}
\newcommand{\T}{\mathcal T}
\renewcommand{\S}{\mathcal S}

\newtheorem{conj}{Conjecture}

\newtheorem{theorem}{Theorem}[section]
\newtheorem{lemma}[theorem]{Lemma}

\newtheorem{cor}[theorem]{Corollary}
\newtheorem{prop}[theorem]{Proposition}
\newtheorem{obs}[theorem]{Observation}

\theoremstyle{definition}

\theoremstyle{remark}

\theoremstyle{definition}
\newtheorem{example}{Example}[section]

\begin{document}

\title{On the size of maximal intersecting families}
\author{Dmitrii Zakharov \thanks{Department of Mathematics, Massachusetts Institute of Technology, Cambridge, MA 02139, USA, Email: {\tt zakhdm@mit.edu}.}}
\date{}

\maketitle

\begin{abstract}
    We show that an $n$-uniform maximal intersecting family has size at most $e^{-n^{0.5+o(1)}}n^n$. This improves a recent bound by Frankl \cite{F}. The Spread Lemma of Alweiss, Lowett, Wu and Zhang \cite{Al} plays an important role in the proof.
\end{abstract}

\section{Introduction}

A family $\F$ of finite sets is called \emph{intersecting} if any two sets from $\F$ have a non-empty intersection. A family $\F$ is called $n$-uniform if every member of $\F$ has cardinality $n$. Suppose that $\F$ is an $n$-uniform intersecting family which is {\it maximal}, i.e. for any $n$-element set $F \not \in \F$ the family $\F \cup \{F\}$ is not intersecting. Note that the ground set of $\F$ is not fixed here, so we allow $F$ to have some elements which do not belong to the support of $\F$.
In 1973, Erd{\H o}s and Lov{\' a}sz \cite{EL} asked how large such a family $\F$ can be. Another way to phrase this question is to ask for the largest size of an $n$-uniform intersecting family $\F$ such that $\tau(\F) = n$. Here, $\tau(\F)$ denotes the \emph{covering number} of the family $\F$, that is, the minimum size of a set $T$ which intersects any member of $\F$. It is easy to see that any such family $\F$ is contained in a maximal intersecting family and any maximal intersecting family $\F$ satisfies $\tau(\F) = n$. 
A related question about the \emph{minimal} size of an $n$-uniform intersecting family $\F$ with $\tau(\F)=n$ was famously solved by Kahn \cite{Kahn}.

In \cite{EL}, Erd{\H o}s and Lov{\' a}sz proved the first non-trivial upper bound $n^n$ on the size of a maximal $n$-uniform intersecting family, and they also constructed such a family of size $[(e -1)n!]$ and conjectured this to be best possible (see also Section \ref{sec4} for the construction). However, 20 years later Frankl, Ota and Tokushige \cite{FOT} gave a new construction of size roughly $(n/2)^n$. The upper bound $n^n$ was improved to $(1 - 1/e +o(1))n^n$ in 1994 by Tuza \cite{Tu}. In 2011, Cherkashin \cite{Ch} obtained a bound $|\F| = O(n^{n-1/2})$ and then in 2017 Arman and Retter \cite{AR} improved this further to $(1+o(1))n^{n-1}$. The best currently known upper bound was obtained in 2019 by Frankl \cite{F}:
\begin{equation}\label{fr}
|\F| \le e^{-c n^{1/4}} n^n.
\end{equation}
Frankl \cite{F} also stated that it is possible to modify the argument and improve the exponent in (\ref{fr}) from $1/4$ to $1/3$. In this paper we provide an even stronger improvement of (\ref{fr}):

\begin{theorem}\label{el}
Let $\F$ be an $n$-uniform maximal intersecting family. Then 
\begin{equation}\label{meq}
    |\F| \le e^{-n^{1/2 + o(1)}} n^n.
\end{equation}
\end{theorem}

Frankl, Ota and Tokushige conjecture in \cite{FOT} that $|\F| \le (\alpha n)^n$ should hold for any maximal intersecting family and some absolute constant $\alpha < 1$. The methods of the present paper do not seem to be sufficient to prove this conjecture. 

To prove Theorem \ref{el}, we consider a more general problem of estimating the number of minimal coverings of an arbitrary intersecting family. Given a family $\F$, a set $T$ is called {\em a minimal covering of} $\F$ if $T \cap F \neq \emptyset$ holds for any $F \in \F$ ($T$ covers $\F$) but this condition does not hold for any proper subset $T' \subset T$ ($T$ is minimal). The minimum size of a covering of $\F$ is called the {\em covering number} and denoted $\tau(\F)$. Let $\T(\F)$ denote the family of all minimal coverings $T$ of a family $\F$. For technical reasons it is convenient to restrict attention to the subfamily $\T_{\le n}(\F) \subset \T(\F)$ of all minimal coverings of $\F$ of size at most $n$ (where $n$ will be taken equal to the uniformity of $\F$). For a not necessarily uniform family $\mathcal G$ and $\lambda > 0$ we define its {\em weight $w_\lambda(\G)$} as follows:
\begin{equation*}
    w_\lambda(\mathcal G) = \sum_{G \in \mathcal G} \lambda^{-|G|}.
\end{equation*}

If $\F$ is an $n$-uniform maximal intersecting family, then $\tau(\F) = n$ and so any element $F \in \F$ is a minimal covering of $\F$. That is, $\F \subset \T(\F)$ and so
\begin{equation}\label{weightbd1}
w_\lambda(\T(\F)) \ge \lambda^{-n} |\F|    
\end{equation}
holds for any $\lambda >0$. On the other hand, the classical encoding procedure of Erd{\H o}s and Lov\'asz \cite{EL} actually shows that any $n$-uniform family $\F$ satisfies 
\begin{equation}\label{weightbd2}
w_n(\T(\F)) \le 1.  
\end{equation}
By putting (\ref{weightbd1}) and (\ref{weightbd2}) together, we recover the upper bound $|\F| \le n^n$. Note that the inequality (\ref{weightbd2}) is actually tight for arbitrary $n$-uniform families: 

\begin{example}
Let $\F = \{F_1, \ldots, F_k\}$ be a collection of $k$ pairwise disjoint $n$-element sets $F_1, \ldots, F_k$; then clearly
$$
\T(\F) = \{T = \{x_1, \ldots, x_k\}:~~ x_i \in F_i, ~i=1, \ldots, k\}
$$ 
and so $w_n(\T(\F)) = |\T(\F)| n^{-k} = 1$.  
\end{example}

However, the family $\F$ in this example is very far from being intersecting. This suggests that perhaps one can improve (\ref{weightbd2}) provided that $\F$ is an intersecting family. Another obstruction comes from the case when $\F$ has small covering number:
\begin{example}
    Let $K_1, \ldots, K_k$ be pairwise disjoint $(n-k+1)$-element sets and let $\F$ be the family of sets of the form $F = K_i \cup T$ where $|T \cap K_j| = 1$ for all $j = 1, \ldots, k$. Then $\F$ is intersecting, $\tau(\F) = \min\{k, n-k+1\}$ and $w_n(\T(\F)) \ge \frac{(n-k+1)^{k}}{n^k} \gtrsim e^{- \frac{k^2}{n} }$. 
\end{example}

So the bound in (\ref{weightbd2}) is essentially tight for $n$-uniform intersecting families $\F$ with covering number $\tau(\F) \lesssim n^{1/2}$.
Our main result states that if $\F$ is intersecting and the covering number $\tau(\F)$ is large enough, then we indeed can win over (\ref{weightbd2}) by a significant amount:

\begin{theorem}\label{main}
For all $\varepsilon > 0$ and sufficiently large $n > n_0(\varepsilon)$ we have the following. Let $\A$ be an intersecting $n$-uniform family. Then
\begin{equation}\label{maine}
c_n(\A) \le e^{1- \frac{\tau(\A)^{1.5-\varepsilon}}{n}}.
\end{equation}
\end{theorem}

Note that this gives a substantial improvement over (\ref{weightbd2}) provided that $\tau(\A) > n^{2/3+\varepsilon}$. 
By applying Theorem \ref{main} to a maximal intersecting family $\F$ and using (\ref{weightbd1}), Theorem \ref{el} follows. 

We now turn to explain the main ideas of the proof of Theorem \ref{main}. In what follows, we use the notation $c_\lambda(\A) = w_\lambda(\T_{\le n}(\A))$ for an $n$-uniform family $\A$ and $\lambda > 0$.

Fix $\varepsilon > 0$. Using induction, we are going to show that for any $n > n_0(\varepsilon)$ and any $n$-uniform intersecting family $\A$ we have
\begin{equation}\label{induction}
    c_n(\A) \le \lambda^{\tau(\A)-n/2},
\end{equation}
where $\lambda = e^{-\frac{1}{n^{0.5+\varepsilon}}}$. This is much weaker than what is claimed in (\ref{maine}) for $n^{2/3} \le \tau(\A) \le n/2$ but gives the same result when $\tau(\A)$ is close to $n$. By writing down the inductive statement (\ref{induction}) more carefully, one can recover (\ref{maine}) in the full range of parameters, see Section \ref{sec3} for details.

If $\tau(\A) \le n/2$ then (\ref{induction}) follows from (\ref{weightbd2}) (which we will prove later) so we may assume that $\tau(\A) \ge n/2$. For the purpose of induction, we may assume that (\ref{induction}) holds for all $n$-uniform intersecting families of size strictly smaller than $\A$. The following proposition is at the core of our inductive approach:

\begin{prop}\label{prop1}
    Let $\lambda = e^{-\frac{1}{n^{0.5+\varepsilon}}}$. If there exists a subfamily $\G \subset \A$ such that $c_{\lambda n}(\G) \le 1$ then $c_{n}(\A) \le \lambda^{\tau(\A)-n/2}$.
\end{prop}

Roughly speaking, Proposition \ref{prop1} tells us that if we can find a subfamily $\G \subset \A$ which is `difficult' to cover then we can use it for the induction step and get a bound on $c_n(\A)$ in terms of $c_n(\A')$ for some proper subfamilies $\A'\subset \A$. The idea of finding a special subfamilies 
in $\A$ to bound the number of minimal covers also appears in a somewhat different form in \cite{F}. 

Proposition \ref{prop1} puts rather strict limitations on how a potential minimal family $\A$ contradicting (\ref{induction}) might look like. 
The first key observation (also originating from \cite{F}) is that all pairwise intersections of sets in $\A$ are either very small or very large. 

Indeed, let $A_1, A_2 \in \A$ be a pair of sets such that $ |A_1 \cap A_2| = k$ for some $k$. Observe that
$$
c_{\lambda n}(\{A_1, A_2\}) = \frac{k}{\lambda n} + \frac{(n-k)^2}{\lambda^2 n^2},
$$
and so we have $c_{\lambda n}(\{A_1, A_2\}) \le 1$ for any $k \in [\sqrt{n}, n-\sqrt{n}]$. So by Proposition \ref{prop1}, unless $\A$ satisfies (\ref{induction}), for any pair $A_1, A_2 \in \A$ we either have $|A_1 \cap A_2| \le \sqrt{n}$ or $|A_1 \cap A_2| \ge n-\sqrt{n}$. 

Let $k = \sqrt{n}$. The above property allows us to write $\A$ as a union
\begin{equation}\label{decompK}
\A = \K_1 \cup \ldots \cup \K_N,
\end{equation}
where for any $i, j = 1, \ldots, N$ and $K_i \in \K_i$ and $K_j \in \K_j$ we have $|K_i \cap K_j| \ge n-k$ if $i=  j$ and $|K_i \cap K_j| \le k$ otherwise. This decomposition step is actually quite robust and works for any $k < n/3$; so if one were to prove (\ref{induction}) with $\lambda < 1-c$ for a small constant $c$, then one may still assume that, say, $|A_1 \cap A_2| \not \in [0.1 n, 0.9 n]$ holds for all $A_1, A_2 \in \A$, and so we have (\ref{decompK}) with $k = 0.1n$. 

The decomposition (\ref{decompK}) has the following properties:

\paragraph{Each family $\K_i$ has a core of size $n-5k$.} That is, there exists a set $K_i$ of size $n-5k$ such that $K_i \subset A$ for any $A \in \K_i$. Note that we only know that $|A_1\cap A_2| \ge n-k$ for any $A_1, A_2 \in \K_i$ and so a priori the sets in $\K_i$ do not have to have a large common intersection. However, if $|\bigcap \K_i|\le n-5k$ then we can take $\G = \K_i$ in Proposition \ref{prop1}: 

\begin{lemma}\label{ker}
Let $k \le  n/10$. Let $\K$ be an $n$-uniform family. Suppose that there is an $(n-k)$-element set $K$ such that we have $|F \cap K| \ge n-2k$ for every $F \in \K$.
Then we either have $c_{n-k}(\K) \le 1$ or $|\bigcap \K| \ge n-5k$.
\end{lemma}

The idea is use the Lubell--Yamamoto--Meshalkin inequality to control possible intersections of a minimal cover of $\K$ with the set $K$ above. This step is also quite flexible and can be employed if one were to prove (\ref{induction}) with $\lambda=1-c$ (and $k \approx c n$).

\paragraph{Each family $\K_i$ is small.} Namely, we have $|\K_i| \le {\tau(\A) + 2k \choose 2k}$ for all $i$. We say that a family $\F$ is $\tau$-critical if removing any set from $\F$ reduces $\tau(\F)$. The family $\A$ in question is $\tau$-critical: if not, then for some proper $\A' \subset \A$ we have $\tau(\A') = \tau(\A)$. But then by the induction assumption we get
$$
c_n(\A) \le c_n(\A') \le \lambda^{\tau(\A')-n/2} = \lambda^{\tau(\A)-n/2}.
$$
Here we also use a simple monotonicity property $c_n(\A) \le c_n(\A')$ which we prove in the next section.

So we can apply the following simple lemma:

\begin{lemma}\label{kerup}
Let $\A$ be a $\tau$-critical $n$-uniform family (that is, removing any element from $\A$ reduces $\tau(\A)$) and let $\K \subset \A$ be a subfamily such that $|\bigcap \K| \ge n-k$ for some $k \ge 0$. 
Then $|\K| \le {\tau(\A) + k \choose k}$.
\end{lemma}

\begin{proof}
Denote $K = \bigcap \K$. 
By $\tau$-criticality of $\A$, for any set $A \in \K$ there is a covering $T_A$ of $\A \setminus \{A\}$ of size less than $\tau(\A)$ which does not intersect $A$. Note that $T_A$ does not intersect $K$ and so it is a covering of the family $(\K \setminus \{A\}) \setminus K$. Thus, the system of pairs of sets $(A\setminus K, T_A)_{A \in \K}$ satisfies the Bollob{\'a}s's Two Families theorem \cite[Page 113, Theorem 8.8]{J} and so $|\K| \le {\tau(\A) + k \choose k}$.  
\end{proof}

\paragraph{The number of families $N$ is small.} Namely, we may assume that $N \le n^{C}$ holds for some constant $C$. This is the part of the proof where we rely on the Spread Lemma of Alweiss--Lovett--Wu--Zhang \cite{Al}. Namely, we have the following:

\begin{lemma}\label{small}
Let $\A$ be an $n$-uniform family where $n$ is sufficiently large. Let $\B \subset \A$ be a subfamily such that $|B_1 \cap B_2| \le k$ for all distinct $B_1, B_2 \in \B$. 
If $k \le \frac{n}{10^4\log n}$ then one of the following 2 possibilities holds:
\begin{enumerate}
    \item We have $|\B| \le n^{C}$ for some absolute constant $C$.
    \item There is a proper subfamily $\A' \subset \A$ such that
$$
2^{\tau(\A)}c_n(\A) \le 2^{\tau(\A')} c_n(\A').
$$
\end{enumerate}
\end{lemma}

Note that this lemma has a mild restriction $k \lesssim \frac{n}{\log n}$. This means that the best possible bound in (\ref{induction}) using Lemma \ref{small} has $\lambda = 1- \frac{c}{\log n}$  (corresponding to a bound of the form $|\F| \le e^{-\frac{cn}{\log n}} n^n$ for maximal intersecting families). So even though this is not enough to prove an exponential bound in the Erd{\H o}s--Lov\'asz problem, this is by far not the main bottleneck of the argument.

The proof of Lemma \ref{small} is based on the following idea. Let $p = \frac{C \log n}{n}$ and consider a random set ${\bf U}$ where each element of the ground set is included in ${\bf U}$ independently with probability $p$. If the family $\T_{\le n}(\A)$ is not $\frac{n}{2}$-spread the one can check that the second option of the lemma holds. Otherwise, by the Spread Lemma (see Lemma \ref{rspread} below), with probability at least $0.9$ there exists an element $T \in \T_{\le n}(\A)$ such that $T \subset {\bf U}$. On the other hand, a routine second moment computation shows that if $N$ is large enough and sets $B_1, \ldots, B_N$ have small pairwise intersections, then with probability at least $0.9$ there exists $i \in [N]$ so that $B_i$ is disjoint from ${\bf U}$. So with probability at least $0.8$ there is a covering $T \subset {\bf U}$ of $\A$ and a set $B_i \in \A$ disjoint from ${\bf U}$. In particular, $T \cap B_i = \emptyset$ with positive probability which contradicts the definition of a covering.

\paragraph{Conclusion: $\A$ is small.} We conclude from the above observations that the family $\A$ itself must be small:
\begin{equation}\label{Asmall}
|\A| \le |\K_1|+\ldots+|\K_N| \le N {\tau(\A)+5k\choose 5k} \le n^{6k}.    
\end{equation}
Once we know that the family $\A$ is small, we can start exploiting the fact that $\tau(\A)$ is large. In fact, we show that $\A$ cannot be too `clustered' around a few elements of the ground set since otherwise we can find a covering of $\A$ of size less than $\tau(\A)$ by sampling a random set according to the degree distribution of $\A$. A careful execution of this idea results in the following lemma:
\begin{lemma}\label{dec}
Let $n\ge 1$ and $m, t \ge 1$. Let $\A$ be an $n$-uniform family of size at most $e^{m}$ and $\tau(\A) \ge t$. Then, for every $l \ge 1$, there is a subfamily $\A' \subset \A$ such that $\tau(\A \setminus \A') \le t/2$ and for every $i = 1, \ldots, l$ we have
\begin{equation}\label{deceq}
    \E_{A_1, \ldots, A_i \in \A'} |A_1 \cap \ldots \cap A_i| \le C_l \left(\frac{m}{t} \right)^{i-1} n,
\end{equation}
where $C_l \ll 2^{l^2}$ depends only on $l$ and the average is taken over all $A_1, \ldots, A_l \in \A$ chosen uniformly and independently.
\end{lemma}

That is, we can remove a few sets from $\A$ and obtain the property that the $l$-wise intersections of sets in $\A$ are very small on average. Note that in our case $m \sim k \log n \sim \sqrt{n} \log n$ and $t = \tau(\A) \ge n/2$, so that 
$$
C_l \left(\frac{m}{t} \right)^{l-1} n \lesssim n^{2- l/2},
$$
that is, almost all $l$-wise intersections of sets from $\A$ are empty for constant $l$. Let $r = n^{0.5-\varepsilon/2}$ and $l = 10\varepsilon^{-1}$ and sample a uniformly random subfamily $\B = \{B_1, \ldots, B_r\} \subset \A'$, where $\A'$ is given by Lemma \ref{dec}. Then by (\ref{deceq}) and the union bound, with positive probability all $l$-wise intersections of sets in $\B$ are empty. We remark that the family $\B$ is a natural generalization of `brooms' used by Frankl in \cite{F}; the advantage of our approach is that we can find (generalized) brooms of size $\sim n^{1/2}$ whereas Frankl could only construct brooms of size $\sim n^{1/4}$.

The final step of the proof is to show that one can take $\G = \B$ in Proposition \ref{prop1}:
\begin{lemma}\label{lbodeg}
Let $n \ge 1$ and $r \ge 2l$ be such that $r^2 \le l^3 n$. Let $\B$ be an $n$-uniform intersecting family of size $r$ such that every $l$ distinct sets from $\B$ have an empty intersection. Then for $k \le \frac{r}{20 l^3}$ we have
\begin{equation*}
    c_{n - k}(\B) \le 1.
\end{equation*}
\end{lemma}

The proof of this lemma crucially uses the intersecting property of the family $\B$. In fact, this is the only place in the argument where we really use the fact that the initial family is intersecting. The construction of a large bounded degree family $\B$ and Lemma \ref{lbodeg} appear to be the main bottlenecks of the argument and are the reason for the resulting bound of $e^{-n^{0.5+\varepsilon}}n^n$ for maximal intersecting families.

The proof of Lemma \ref{lbodeg} is based on a more careful analysis of the classical Erd{\H o}s--Lov\'asz encoding procedure: the intersecting property and bounded degree of $\B$ ensure that there is enough `overlap' between sets $B_i$ which makes the encoding more efficient. This completes the proof of Theorem \ref{main}. The next section contains all the proofs of the lemmas which appeared in this outline and in Section \ref{sec3} we formally deduce Theorem \ref{main}. Section \ref{sec4} contains some final remarks and questions.

\section{Proving auxiliary results}\label{sec2}

\subsection{Minimal covers}\label{sec21}

Fix $n \in \N$ and let $\A$ be a finite family of sets of size at most $n$. For $\lambda > 0$, we define the \emph{weight} $w_\lambda(\A)$ of $\A$ by the following expression:
\begin{equation*}
    w_\lambda(\A) = \sum_{A \in \A} \lambda^{-|A|}.
\end{equation*}

The parameter $\lambda$ will be usually taken to be $\lambda = n$ or $\lambda = n - k$ for a relatively small number $k$. 
The following characteristic of a family will be crucial for our arguments.  Recall that a covering of $\A$ is a set $T$ intersecting all members of $\A$ and a covering $T$ is minimal if any proper subset $T' \subset T$ does not cover $\A$. We denote $\tau(\A)$ the minimum size of a 
covering of $\A$.  
We denote $\T_{\le n}(\A)$ the family of all minimal covers of $\A$ of size at most $n$. For $\lambda > 0$, put 
$$
c_\lambda(\A) = w_\lambda(\T_{\le n}(\A)) = \sum_{T \in \T_{\le n}(\A)} \lambda^{-|T|}.
$$
We remark that the family $\T_{\le n}(\F)$ was also introduced in \cite{Tu} to prove the bound $|\F| \le (1 - e^{-1} + o(1))n^n$. 
We have the following basic monotonicity result:

\begin{obs}\label{mon}
For any family $\F$ and all $\lambda \le \mu$ we have 
$$
c_\mu(\F) \le \left( \frac{\lambda}{\mu} \right)^{\tau(\F)} c_\lambda(\F).
$$
\end{obs}

\begin{proof}
Indeed, since every minimal covering $T$ of $\F$ has size at least $\tau(\F)$ we have 
$$
\mu^{-|T|} \le \left( \frac{\lambda}{\mu} \right)^{\tau(\F)} \lambda^{-|T|}.
$$
Summing over all $T \in \T_{\le n}(\F)$ gives the desired inequality.
\end{proof}

Let $X$ be the ground set of $\A$. For $S \subset X$ we denote by $\A(\bar S)$ the set of elements of $\A$ which do not intersect $S$. The following lemma lies at the foundation of our arguments.

\begin{lemma}\label{lm}
For any subfamily $\A'\subset \A$ of any family $\A$ and for any $\lambda > 1$ we have
\begin{equation}\label{ineqmain}
c_\lambda(\A) \le \sum_{T' \in \T_{\le n}(\A')} \lambda^{-|T'|} c_\lambda(\A(\bar T')),    
\end{equation}
In particular, we have
$$
c_\lambda(\A) \le c_\lambda(\A') \max_{T' \in \T_{\le n}(\A')} c_\lambda(\A(\bar T')).
$$
\end{lemma}

\begin{proof}
Each minimal covering $T \in \T_{\le n}(\A)$ contains a minimal covering $T' \subset T$ of $\A'$. Moreover, by the minimality of $T$, the set $T\setminus T'$ is a minimal covering of the family $\A(\bar T')$. So each term $\lambda^{-|T|}$ on the left hand side of (\ref{ineqmain}) corresponds to at least one term $\lambda^{-|T'|} \lambda^{-|T\setminus T'|}$ on the right hand side of (\ref{ineqmain}) (there could be more than one way to choose $T'$). This proves (\ref{ineqmain}).
\end{proof}

In particular, we have:

\begin{cor}[Tuza, \cite{Tu}]\label{st}
For any $n$-uniform family $\F$ we have $c_n(\F) \le 1$.
\end{cor}

This bound was also proved in \cite{Tu}, similar ideas appear in \cite{G}.

\begin{proof}
Note that if $|\F| \le 1$ then the proposition holds.
If $|\F| \ge 2$ then choose a proper non-empty subfamily $\F' \subset \F$ and apply the second part of Lemma \ref{lm}. The statement now follows by induction.
\end{proof}

The basic idea behind the proof of Theorem \ref{el} is to use apply Lemma \ref{lm} to various subfamilies $\F'$ with small $c_n(\F')$ and use induction to estimate the terms $c_n(\F(\bar T'))$. More precisely, we will use the following consequence of Lemma \ref{lm}.




\begin{lemma}\label{crlm}
    Let $f: \R\rightarrow \R$ be a differentiable convex function. Let $\F$ be an $n$-uniform family such that for any proper subfamily $\A \subset \F$ we have $c_n(\A) \le e^{- f(\tau(\A))}$. Let $\lambda = e^{-f'(\tau(\F))}$ and suppose that there exists a non-empty family $\F' \subset \F$ such that $c_{\lambda n}(\F') \le 1$. Then $c_n(\F) \le e^{-f(\tau(\F))}$.
\end{lemma}

\begin{proof}
    By Lemma \ref{lm} applied to $\F' \subset \F$ we have 
\begin{equation*}
    c_n(\F) \le \sum_{T \in \T_{\le n}(\F')} n^{-|T|} c(\F(\bar T)) \le \sum_{T \in \T_{\le n}(\F')} n^{-|T|} e^{-f(\tau(\F(\bar T)))}. 
\end{equation*}
We have $\tau(\F(\bar T)) \ge \tau(\F) - |T|$ and so by convexity $f(\tau(\F(\bar T))) \ge f(\tau(\F)) - |T| f'(\tau(\F))$, which leads to
$$
c_n(\F) \le \sum_{T \in \T_{\le n}(\F')} n^{-|T|} e^{f'(\tau(\F)) |T| - f(\tau(\F))} = c_{\lambda n}(\F') e^{- f(\tau(\F))} \le e^{- f(\tau(\F))},
$$
completing the proof.
\end{proof}

Note that Proposition \ref{prop1} from the proof outline above follows from this lemma with $f(t) = - (t-n/2) \log \lambda$.

\subsection{Large intersections}\label{sec22}

In this section we study families $\K$ in which every pair of sets has ``large" intersection. 

\begin{lemma}\label{ker2}
Let $k \le  n/10$. Let $\K$ be an $n$-uniform family. Suppose that there is an $(n-k)$-element set $K$ such that we have $|F \cap K| \ge n-2k$ for every $F \in \K$.
Then we either have $c_{n-k}(\K) \le 1$ or $|\bigcap \K| \ge n-5k$.
\end{lemma}

\begin{proof}
Let $K' = \bigcap \K$ and $R = K \setminus K'$ and let $u = |K' \setminus K|$. Note that $u \in [0, k]$ since $|F \setminus K| \le k$ for all $F \in \K$.
Denote by $\A$ the family of all sets $F \setminus K'$ for $F \in \K$. By the definition of $\A$ we have $\tau(\A) \ge 2$. Note also that for any $A \in \A$ we have $|A \setminus R| \le  k-u$ (since $A$ has size at most $k$ and contains the $u$-element set $K'\setminus K$). 

Note that a minimal covering $T$ of the family $\K$ is either contained in $K'$ and $|T| = 1$ or $T \cap K' = \emptyset$. In the latter case $T$ is obviously a minimal covering of $\A$. Thus, we have
\begin{equation}\label{ckca}
c_\lambda(\K) = \frac{|K'|}{\lambda} + c_\lambda(\A).
\end{equation}

Let $\T_1 \subset \T_{\le n}(\A)$ be the family of minimal coverings $T$ of $\A$ which are subsets of $R$. Let $\T_2 = \T_{\le n}(\A) \setminus \T_1$. We will estimate weights of $\T_1$ and $\T_2$ separately. 

Note that $\T_1 \subset 2^R$ and observe that $T' \not\subset T$ for any distinct $T, T' \in \T_1$ (i.e. $\T_1$ is an antichain in $2^R$). 

\begin{prop}\label{sp}
Suppose that $\T \subset 2^R$ is an antichain such that every element of $\T$ has size at least $t$. If $\lambda \ge |R|$ then
$$
\sum_{T \in \T} \lambda^{-|T|} \le \lambda^{-t} {|R| \choose t}.
$$
\end{prop}

This statement also appears in \cite{F0}.

\begin{proof}
Note that for any $s \ge t$ we have ${|R| \choose s} \le {|R| \choose t} \lambda^{s-t}$ and so by the Lubell–Yamamoto–Meshalkin inequality \cite[Page 112, Theorem 8.6]{J}:
$$
\sum_{T \in \T} \lambda^{-|T|} \le \sum_{T \in \T} \lambda^{-t}{|R| \choose t} / {|R| \choose |T|} = \lambda^{-t}{|R| \choose t} \sum_{T \in \T} \frac{1}{{|R| \choose |T|}} \le \lambda^{-t} {|R| \choose t}.
$$
\end{proof}

By Lemma \ref{sp} for every $\lambda \ge |R|$ the $\lambda$-weight of $\T_1$ is at most 
\begin{equation}\label{c1}
    w_\lambda(\T_1) \le \lambda^{-\tau(\A)} {|R| \choose \tau(\A)} \le  \frac{(|R|/\lambda)^{\tau(\A)}} {\tau(\A)!}.
\end{equation}
Now we estimate the weight of $\T_2$. Let $\S \subset 2^R$ be the family of all sets $S \subset R$ such that $S$ {\it does not} cover $\A$. Then the weight of $\T_2$ is bounded by the following expression:
\begin{equation}\label{exp}
    w_\lambda(\T_2) \le \sum_{S \in \S} \lambda^{-|S|} c_\lambda(\A(\bar S) \setminus R).
\end{equation}
Indeed, the contribution of an element $T \in \T_2$ on the left hand side is accounted by the term corresponding to $S = T\cap R \in \S$ on the right hand side (since $T\setminus R$ is a minimal covering of the family $\A(\bar S)\setminus R$).
Here the family $\A(\bar S) \setminus R$ consists of all sets of the form $A \setminus R$ where $A \in \A$ does not intersect $S$. Every element in $\A(\bar S)\setminus R$ has cardinality at most $k-u$ and so by Observation \ref{mon} and Corollary \ref{st} applied to $\A(\bar S)\setminus R$ for every $\lambda \ge k-u$ we have
\begin{equation}\label{eql}
    c_\lambda(\A(\bar S)\setminus R) \le \left(\frac{k-u}{\lambda}\right)^{\tau(\A(\bar S)\setminus R)} c_{k-u}(\A(\bar S)\setminus R) \le \left(\frac{k-u}{\lambda}\right)^{\tau(\A(\bar S)\setminus R)}.
\end{equation}

Let $S \in \S$. Note that we have the following lower bound on $\tau(\A(\bar S)\setminus R)$:
\begin{equation*}
    \tau(\A(\bar S)\setminus R) \ge \max \{ 1, \tau(\A)-|S|\}.
\end{equation*}
Using this lower bound, (\ref{exp}) and (\ref{eql}) we obtain an upper bound on the weight of $\T_2$ for $\lambda \ge k$:
\begin{align}\label{wc2}
    w_\lambda(\T_2) \le \sum_{s = 0}^{\tau(\A)-1} \lambda^{-s} {|R| \choose s} \left(\frac{k-u}{\lambda}\right)^{\tau(\A) - s} +  \sum_{s = \tau(\A)}^{|R|} \lambda^{-s} {|R| \choose s}\left(\frac{k-u}{\lambda}\right).
\end{align}

Now we combine all obtained inequalities to prove Lemma \ref{ker}. Suppose that $|K'| = |\bigcap \K| < n-5k$, we need to show that $c_{n-k}(\K) \le 1$ holds. We have 
$$
|K\cup K'| = |K'| + |K\setminus K'|= |K| + |K'\setminus K|,
$$
so that $|K'| = n-k + u - r$ holds. In particular, by the assumption $|K'|< n-5k$ we have $n-k \ge r > u+4k$.

Denote $t = \tau(\A) \ge 2$, $r = |R|$, $\rho = \frac{r}{n-k}$ and $\delta = \frac{k-u}{n-k}$.
We can use (\ref{c1}) and (\ref{wc2}) with $\lambda = n-k$ and get:
\begin{align*}
    w_{n-k}(\T_1) &\le \frac{\rho^{t}}{t!},\\
    w_{n-k}(\T_2) &\le \sum_{s=0}^{t-1} \frac{\rho^s \delta^{t-s}}{s!} + \sum_{s=t}^{r} \frac{\rho^s \delta}{s!}. 
\end{align*}

By (\ref{ckca}), formula $|K'| = n-k+u-r$ and decomposition $\T_{\le n}(\K) = \T_1\cup \T_2$:
\begin{align*}
c_{n-k}(\K) &\le \frac{n-k+u-r}{n-k} + w_{n-k}(\T_1) + w_{n-k}(\T_2) \le\\
    & \le \frac{n-k+u-r}{n-k} + \frac{\rho^t}{t!} + \sum_{s=0}^{t-1} \frac{\rho^s \delta^{t-s}}{s!} + \sum_{s=t}^{r} \frac{\rho^s \delta}{s!}. 
\end{align*}
Both $\rho$ and $\delta$ are between 0 and 1 so it is easy to see that the second line is the largest when $t=2$, i.e.
$$
c_{n-k}(\K) \le \frac{n-k+u-r}{n-k} + \frac{\rho^2}{2} + \delta^2 + \frac{\delta \rho}{2} + \sum_{s=2}^{r} \frac{\rho^s \delta}{s!} \le \frac{n-k+u-r}{n-k} + \frac{\rho}{2} + 2\delta,
$$
where in the last transition we used $0\le \delta, \rho \le 1$ to group the last 3 terms together and replace $\rho^2$ by $\rho$. Recalling $\rho = \frac{r}{n-k}$ and $\delta = \frac{k-u}{n-k}$ we get
$$
c_{n-k}(\K) \le \frac{n+ k - u - r/2}{n-k} \le \frac{n+k - r/2}{n-k} \le 1,
$$
since $r \ge 4k$ and $u \ge 0$.
\end{proof}

\subsection{Small intersections}

In this section we show that in some cases it is possible to estimate the size of a subfamily $\B \subset \A$ provided that elements of $\B$ have very small pairwise intersections.

\begin{lemma}\label{small2}
Let $\A$ be an $n$-uniform family where $n$ is sufficiently large. Let $\B \subset \A$ be a subfamily such that $|B_1 \cap B_2| \le k$ for all distinct $B_1, B_2 \in \B$. 
If $k \le \frac{n}{10^4\ln n}$ then one of the following 2 possibilities holds:
\begin{enumerate}
    \item We have $|\B| \le n^{C}$ for some absolute constant $C$.
    \item There is a proper subfamily $\A' \subset \A$ such that 
$$
2^{\tau(\A)}c_n(\A) \le 2^{\tau(\A')} c_n(\A').
$$
\end{enumerate}
\end{lemma}

To prove this lemma we will need a result on $R$-spread families which was recently used to substantially improve the upper bound in the Erd{\H o}s-Rado Sunflower problem \cite{Al}, \cite{R}. We will use a variant of this result proved in \cite[Corollary 7]{Ta}. Let $\bf C$ be a random set, that is a probability distribution on $2^X$ for some finite ground set $X$. For $R \ge 1$ we say that $\bf C$ is {\it an $R$-spread random set} if for every set $S \subset X$ the probability that $\bf C$ contains $S$ is at most $R^{-|S|}$.

\begin{lemma}[\cite{Ta}]  \label{rspread}
Let $R > 1$, $\delta \in (0, 1)$ and $m \ge 1$. Let $\bf C$ be an $R$-spread random subset of a finite set $X$. Let ${\bf W} \subset X$ be a random set independent from $\bf C$ and such that each $x \in X$ belongs to $\bf W$ with independent probability $1 - (1-\delta)^m$. Then there exists a random set $\bf C'$ with the same distribution as $\bf C$ and such that 
\begin{equation*}
\E |{\bf C}' \cap {\bf W}| \le \left(  \frac{5}{\log_2 R\delta} \right)^m \E |{\bf C}|.
\end{equation*}
\end{lemma}

We will in fact only need the following corollary of this result.

\begin{cor}\label{cors}
In the notations of Lemma \ref{rspread}, let $\C \subset 2^X$ be the support of the random set $\bf C$. Then the probability that a random set $\bf W$ of density $1 - (1-\delta)^m$ contains an element of $\C$ is at least 
\begin{equation*}
    \P(\exists C \in \C:~ C \subset {\bf W}) \ge 1 - \left(  \frac{5}{\log_2 R\delta} \right)^m \E |{\bf C}|.
\end{equation*}
\end{cor}

\begin{proof}[Proof of Lemma \ref{small}]

Denote by $X$ the ground set of $\A$. 
Put $C = 2048$, $R=\frac{n}{2}$ and $m = \lceil \log_2{n}+10\rceil$. Let $\delta = \frac{C}{n}$ and let ${\bf U} \subset X$ be a subset of $X$ of density $(1-\delta)^m$. Let ${\bf T} \in \T_{\le n}(\A)$ be a random set with distribution 
\begin{equation}\label{distri}
\P({\bf T} = T) = \frac{n^{-|T|}}{c_n(\A)},    
\end{equation}
where $T \in \T_{\le n}(\A)$ and such that ${\bf T}$ is independent from $\bf U$.

Let us suppose that the random set $\bf T$ is not $R$-spread. By definition, this means that there is a non-empty set $S \subset X$ such that
$$
\P(S \subset {\bf T}) \ge R^{-|S|} = \left( \frac{2}{n} \right)^{|S|}.
$$
Let $\A' = \A(\bar S)$ be the family of $A \in \A$ such that $A \cap S = \emptyset$. Note that $\tau(\A')\ge \tau(\A) - |S|$. Note that if a covering $T \in \T_{\le n}(\A)$ satisfies $S \subset T$ then $T\setminus S$ is a minimal covering of the family $\A'$. Thus, 
$$
\sum_{T \in \T_{\le n}(\A):~ S \subset T} n^{-|T|} = n^{-|S|} \sum_{T \in \T_{\le n}(\A): ~S \subset T} n^{-|T \setminus S|} \le n^{-|S|} c_n(\A').
$$
By (\ref{distri}), the left hand side of this inequality equals to $c_n(\A) \P(S \subset {\bf T})$. We conclude
\begin{align*}
    c_n(\A') n^{-|S|} \ge c_n(\A) \P(S \subset {\bf T}) \ge c_n(\A) 2^{|S|} n^{-|S|},\\
    c_n(\A') \ge c_n(\A) 2^{|S|} \ge c_n(\A) 2^{\tau(\A) - \tau(\A')}.
\end{align*}
This implies the second alternative of Lemma \ref{small2}. So we may assume that $\bf T$ is $R$-spread.

By Corollary \ref{cors} (applied to ${\bf W} = X\setminus {\bf U}$), we have the following estimate on the probability that there is a covering $T \in \T_{\le n}(\A)$ which does not intersect $\bf U$:
$$
\P(\exists T \in \T_{\le n}(\A):~ T \cap {\bf U} = \emptyset) \ge 1 - \left(  \frac{5}{\log R\delta} \right)^m \E |{\bf T}| \ge 1 - 2^{-m} n > 0.9,
$$
here we used the fact that $R\delta = 1024$, $m > \log_2 n + 9$ and that every element of $\T_{\le n}(\A)$ has size at most $n$.\footnote{In fact, this is the only place in the argument where we need this restriction on the sizes of the coverings.} 
We conclude that if we take a random set $\bf U$ of density $(1-\delta)^m$ then with probability at least $0.9$ there is a $T \in \T_{\le n}(\A)$ which does not intersect $\bf U$. 
Let us now show that with probability at least $0.5$ the set $\bf U$ contains an element of $\B$, provided that $\B$ is large enough. Since by definition of $\T_{\le n}(\A)$ every $T \in \T_{\le n}(\A)$ intersects every set from $\B$, this will lead to a contradiction if $|\B|$ is large enough.

Note that an element $A$ of $\B$ is contained in $\bf U$ with probability 
\begin{equation}\label{rh}
    (1-\delta)^{m n} = e^{n \log_2 n (-\frac{C}{n} + O(n^{-2}))} = n^{-C/\ln 2 + o(1)},
\end{equation}
provided that $n$ is sufficiently large. Denote $\rho = (1-\delta)^{n m}$. 
For $A \in \B$ denote by $\xi_A$ the indicator of the event that $A \subset \bf U$ and by $\xi$ the sum of $\xi_A$ over $\B$. Hence, we have $\E \xi_A = \rho$ for every $A \in \B$ and
$$
\E \xi = |\B| \rho. 
$$
By Chebyshev's inequality (see, for instance, \cite[Page 303, (21.2)]{J}), it is enough to show that ${\rm Var}\, \xi < (\E \xi)^2 / 2$, where ${\rm Var}\, \xi$ denotes the variance of the random variable $\xi$. Let us estimate the correlations $(\E \xi_A \xi_{A'} - \rho^2)$ for $A \neq A'$. It is clear that 
$$
\E \xi_A \xi_{A'} = (1-\delta)^{m|A\cup A'|} \le (1-\delta)^{2 m n - m\frac{n}{10^4 \ln n} } = \rho^{2 - \frac{1}{10^4 \ln n}}.
$$
By (\ref{rh}), we have
$$
\rho^{-\frac{1}{10^4 \ln n}} = \left(n^{\frac{C + o(1)}{\ln 2} }\right)^{- \frac{1}{10^4 \ln n}} = 2^{\frac{C}{10^4} + o(1)} < 1.4
$$
provided that $n$ is large enough. We conclude that the variance of $\xi$ is at most
$$
0.4 \rho^2 |\B|^2 + \rho |\B|
$$
which is less than $(\E\xi)^2/2$ if $|\B| > 10/\rho$. Therefore, provided, that $|\B| \ge n^{3000} > 10/\rho$, with probability at least $0.5$ the random set $\bf U$ contains an element of $\B$ and with probability at least $0.9$ it does not intersect an element of $\T_{\le n}(\A)$. But these two events cannot happen simultaneously. This is a contradiction and Lemma \ref{small2} is proved. 
\end{proof}

\subsection{Moments of the degree function}

In this section we show that if we have an $n$-uniform family $\A$ such that $\tau(\A)$ is ``large" but $|\A|$ is ``small" then the $l$-wise intersections of sets from $\A$ are very small on average. More precisely, we will prove the following:

\begin{lemma}\label{dec2}
Let $n\ge 1$ and $m, t \ge 1$. Let $\A$ be an $n$-uniform family of size at most $e^{m}$ and $\tau(\A) \ge t$. Then, for every $l \ge 1$, there is a subfamily $\A' \subset \A$ such that $\tau(\A \setminus \A') \le t/2$ and for every $i = 1, \ldots, l$ we have
\begin{equation*}
    \E_{A_1, \ldots, A_i \in \A'} |A_1 \cap \ldots \cap A_i| \le C_l \left(\frac{m}{t} \right)^{i-1} n,
\end{equation*}
where $C_l \ll 2^{l^2}$ depends only on $l$ and the average is taken over all $A_1, \ldots, A_l \in \A$ chosen uniformly and independently.
\end{lemma}

Let $X$ denote the ground set of an $n$-uniform family $\F$. For a function $f: X \rightarrow \R_+$ and $S \subset X$ we denote by $f(S)$ the sum $\sum_{x\in S}f(x)$. 

\begin{obs}\label{sa}
For any non-zero function $f: X \rightarrow \R_+$ and any family $\F$ on $X$ we have
\begin{equation}\label{dis}
    \sum_{F \in \F} \left(1 - \frac{f(F)}{f(X)}  \right)^{\tau(\F)-1} \ge 1.
\end{equation}
In particular, for any $f: X \rightarrow \R_+$ there always exists $F \in \F$ such that
$$
f(F) \le f(X) (1 - |\F|^{-1/(\tau(\F)-1)}).
$$
\end{obs}

\begin{proof}
Put $t = \tau(\F) - 1$ and let $x_1, \ldots, x_t \in X$ be a sequence of random independent elements of $X$ sampled according to distribution $f$.
Then the left hand side of (\ref{dis}) is the expectation of the number of sets $F \in \F$ which are not covered by the set $\{x_1, \ldots, x_t\}$. Since $\tau(\F) > t$, this number is always positive and (\ref{dis}) follows.
\end{proof}

The following variant of this observation will be slightly more convenient to use.

\begin{cor}\label{mu}
Let $f_1, \ldots, f_l: X \rightarrow \R_+$ be arbitrary non-zero functions and $\F$ be an arbitrary family on $X$. Then there exists $F \in \F$ such that $f_i(F) \le f_i(X) l (1 - |\F|^{-1/(\tau(\F)-1)})$ for any $i = 1, \ldots, l$.
\end{cor}

\begin{proof}
Apply Observation \ref{sa} to $f(x) = \sum_{i = 1}^l \frac{f_i(x)}{f_i(X)}$.
\end{proof}

For a family $\F$ on the ground set $X$ let $d_\F: X \rightarrow \R_+$ be the degree function of the family $\F$, that is, if $x \in X$ then $d_\F(x)$ equals to the number of sets $F \in \F$ which contain $x$. Let $d_\F^l: X \rightarrow \R_+$ denote the $l$-th power of $d_\F$, i.e. $d_\F^l(x) = (d_\F(x))^l$. By abusing notation, we also denote by $d^l_\F$ the number $d^l_\F(X)$.  

\begin{obs}
For any family $\F$ and any $l \ge 1$ we have the following identity
$$
d_\F^l |\F|^{-l} = \E_{F_1, \ldots, F_l \in \F} |F_1 \cap \ldots \cap F_l|,
$$
where $F_1, \ldots, F_l$ are taken from $\F$ uniformly and independently.
\end{obs}

Applying Corollary \ref{mu} to functions $d_\F^1, \ldots, d_\F^l$ we obtain the following result.

\begin{lemma}\label{in}
Let $l, t \ge 1$, let $\F \subset \A$ be a subfamily of a family $\A$ such that $\tau(\A \setminus \F) \ge t+1$. Then there exists $A \in \A \setminus \F$ such that the following holds. Denote $\F' = \F \cup \{A\}$, then for any $i = 1, \ldots, l$ we have:
\begin{equation}\label{req}
    d_{\F'}^i \le d_\F^i + \left(\frac{l \log |\A|}{t}\right) 2^i d_\F^{i-1} + n.
\end{equation}
\end{lemma}

\begin{proof}
For $i = 1, \ldots, l$ let 
\begin{equation}\label{f}
    f_i(x) = \sum_{j = 1}^{i-1} {i \choose j} d_\F^j(x).
\end{equation}
Apply Corollary \ref{mu} to functions $f_1, \ldots, f_l$ and the family $\A \setminus \F$. Then there exists $A \in \A\setminus \F$ such that for every $i = 1, \ldots, l$ we have 
\begin{equation}\label{ineq}
    f_i(A) \le f_i(X) l (1 - |\A \setminus \F|^{-1/\tau(\A\setminus \F)-1}) \le f_i(X) l(1 - |\A|^{-1/t}) \le f_i(X) \frac{l \log |\A|}{t},
\end{equation}
by the standard inequality $1-e^{-x} \le x$. But note that for $\F' = \F \cup \{A\}$ by (\ref{f}) we have 
\begin{equation}\label{dmd}
    d_{\F'}^i - d_{\F}^i = \sum_{x \in X} d_{\F'}^i(x) - d_\F^i(x) = \sum_{x \in A} (d_{\F}(x) + 1 )^i - d_\F^i(x) = f_i(A) + n.
\end{equation}
Note that $d_\F^i$ is monotone increasing in $i$ and so
$$
f_i(X) = \sum_{j = 1}^{i-1} {i \choose j} d_\F^j \le 2^i d_\F^{i-1}.
$$
The bound (\ref{req}) now follows from (\ref{ineq}) and (\ref{dmd}).
\end{proof}

Now we are ready to prove Lemma \ref{dec2}.

\begin{proof}[Proof of Lemma \ref{dec2}]

Let $X$ denote the ground set of $\A$ and put $\gamma = 2l m /t$.

Let $\F \subset \A$ be a maximal subfamily in $\A$ such that for every $i = 1, \ldots, l$ we have
\begin{equation}\label{max}
    d_\F^i \le 2^{i^2}\gamma^{i-1} n |\F|^i + 2^{i^2}n|\F|.
\end{equation}
Note that if $|\F| = 1$ then (\ref{max}) clearly holds and so $\F$ is well-defined. To prove Lemma \ref{dec} it is clearly enough to show that any such $\F$ satisfies $\tau(\A \setminus \F) \le t / 2$. Indeed, in this case we have $\tau(\F) \ge t/2$ and, in particular, $|\F| \ge t/2$.  
Then $\gamma |\F| \ge m \ge 1$ and, therefore, the first term in (\ref{max}) dominates the second one.

Now we show that it is impossible to have $\tau(\A \setminus \F) \ge t/2 + 1$. Indeed, in this case we can apply Lemma \ref{in} to the pair $\F \subset \A$ and obtain a family $\F' = \F \cup \{ A \}$ such that (\ref{req}) holds for $i = 1, \ldots, l$ and with $t/2$ instead of $t$. Note that $l \log |\A|/(t/2) \le \gamma$. On the other hand, the maximality of $\F$ implies that there is some $i \in \{2, \ldots, l\}$ ($i \neq 1$ because otherwise (\ref{max}) holds automatically) such that 
$$
d^i_{\F'} > 2^{i^2} \gamma^{i-1}n(|\F|+1)^i + 2^{i^2}n(|\F|+1) \ge 2^{i^2}\gamma^{i-1} n|\F|^i + 2^{i^2}\gamma^{i-1}n|\F|^{i-1} + 2^{i^2}n|\F| + 2^{i^2}n.
$$
On the other hand, from (\ref{req}) we get
\begin{align*}
    d^i_{\F'} \le d_\F^i + \gamma 2^i d_\F^{i-1} + n \le  (2^{i^2}\gamma^{i-1} n|\F|^i + 2^{i^2}n|\F|) + \gamma 2^i( 2^{(i-1)^2}\gamma^{i-2} n|\F|^{i-1} + 2^{(i-1)^2} n |\F|) + n.
\end{align*}
Combining these two inequalities and cancelling same terms we get
$$
(2^{i^2} - 2^{i^2 - i + 1})\gamma^{i-1}n|\F|^{i-1} - \gamma 2^{i^2 - i + 1} n|\F| + (2^{i^2} - 1)n < 0.
$$
So if we let $x = \gamma |\F|$ then, after dividing by $2^{i^2}n$, we obtain
\begin{equation}\label{ss}
2^{-i+1}x > (1 - 2^{-i+1})x^{i-1}+\frac{1}{2}.
\end{equation}
Recall that $i \ge 2$. So if $x \ge 1$ then the first term on the right hand side (\ref{ss}) is greater than $2^{-i+1}x$. If $x < 1$ then the second term is greater than $2^{-i+1}x$. In both cases we arrive at a contradiction. Lemma \ref{dec} is proved.
\end{proof}

\subsection{Bounded degree families}\label{sec25}

In this section we consider intersecting families of bounded degree. In fact, this is essentially the only place in the paper where we use the fact that the family is intersecting. The idea to consider low degree families in the Erd{\H o}s--Lov{\' a}sz problem also appears in \cite[Section 2]{F}.

\begin{lemma}\label{lbodeg2}
Let $n \ge 1$ and $r \ge 2l$ be such that $r^2 \le l^3 n$. Let $\B$ be an $n$-uniform intersecting family of size $r$ such that every $l$ distinct sets from $\B$ have an empty intersection. Then 
\begin{equation}\label{lel}
    c_n(\B) \le e^{-\frac{r^2}{10 l^3 n}}.
\end{equation}
\end{lemma}


\begin{proof}[Proof of Lemma \ref{lbodeg2}]
In order to prove this lemma, we need to recall the classical Er{\H o}s--Lov{\' a}sz encoding procedure which they used to obtain the bound $|\F| \le n^n$ for the size of an $n$-uniform maximal intersecting family.
Denote $\B = \{F_1, \ldots, F_r\}$. 

\paragraph{Procedure.} Let $T \in \T_{\le n}(\B)$ and $S \subset T$ be a proper subset. From the pair $(T, S)$ we construct a new pair $(T, S')$ as follows. Let $i \in [r]$ be the minimum number so that $F_i \cap S = \emptyset$. Pick arbitrary $x \in F_i \cap T$ and let $S' = S \cup \{x\}$. 

So if we apply this procedure to any $T \in \T_{\le n}(\B)$ and $S = \emptyset$ then we will obtain a sequence of sets of the form:
\begin{equation}\label{seq}
    \emptyset = S_0 \subset S_1 \subset \ldots \subset S_{|T|} = T.
\end{equation}

Note that the sequence $(S_0, \ldots, S_{|T|})$ is not determined uniquely by $T$ since there may be an ambiguity in the choice of $x \in F_i \cap T$ during the procedure. Let $\T_1 \subset \T_{\le n}(\B)$ the the family of sets $T$ such that the sequence $(S_0, \ldots, S_{|T|})$ is determined uniquely by $T$. In other words, at each step we have an equality $|F_i \cap T| = 1$. Let $\T_2 = \T_{\le n}(\B) \setminus \T_1$.

Now we denote by $\mathcal J$ the set of all sequences $(S_0, S_1, \ldots, S_k)$ which may occur during the procedure starting from some $T \in \T_{\le n}(\B)$ and $S = \emptyset$. Let $\J = \J_1 \cup \J_2$ be the decomposition arising from the decomposition $\T_{\le n}(\B) = \T_1 \cup \T_2$. The weight $w(\bar S)$ of a sequence $\bar S = (S_0, \ldots, S_k)$ is defined to be $n^{-|S_k|}$. The standard Erd{\H o}s--Lov{\' a}sz \cite{EL} argument shows that the weight $w(\J)$ of the family $\J$ is always at most $1$. We omit the proof since it is very similar in spirit to the proof of Lemma \ref{lm} and Corollary \ref{st}. 

On the other hand, we can bound weights of families $\T_1$ and $\T_2$ in terms of weights of $\J_1$ and $\J_2$ as follows:
\begin{equation}\label{cj}
    c_n(\B) = w_n(\T_{\le n}(\B)) = w_n(\T_1) + w_n(\T_2) \le w(\J_1) + \frac{1}{2} w(\J_2) \le \frac{w(\J_1) + 1}{2}.
\end{equation}
So it is enough to obtain a good upper bound on $w(\J_1)$. For $T \in \T_1$ we denote by $S_i(T)$ the $i$-th element of the sequence of $T$ in the process (which is defined uniquely for elements of $\T_1$). We denote by $A_i(T) \in \B$ the element of $\B$ which was picked at step $i-1$ of the process. In particular, $S_{i-1}(T) \cap A_i(T) = \emptyset$ and $|S_{i}(T) \cap A_i(T)| = 1$. We denote by $x_i(T)$ the unique element in the intersection $S_{i}(T) \cap A_i(T)$. 

The uniqueness of the sequence $\bar S(T)$ implies that for any $j < i$ we have
$$
x_i(T) \not \in A_j(T).
$$
Indeed, otherwise at step $j$ we may have picked the element $x_i(T)$ instead of $x_j(T)$ and thus form a different sequence $(S_0', \ldots, S'_{|T|})$ which corresponds to the covering $T$. We conclude that
$$
x_i(T) \in A_i(T) \setminus \bigcup_{j < i} A_j(T) =: Y_i(T).
$$
Since the family $\B$ is intersecting and does not contain $l$-wise intersections we have the following upper bound on the size of $Y_i(T)$:
$$
|Y_i(T)| \le n - \frac{i-1}{l}.
$$
For $q \ge 0$ and a given sequence $\bar S = (S_0 \subset S_1 \subset \ldots \subset S_q)$ we denote by $\J_1(\bar S)$ the family of sequences from $\J_1$ which start from $\bar S$. 

\begin{obs}\label{obb}
For any sequence $\bar S = (S_0, S_1, \ldots, S_{i-1})$ which is a part of the sequence of some $T\in \T_1$ such that $|T| > i$ we have 
$$
w(\J_1(\bar S)) \le \frac{1}{n} \sum_{x \in Y_i(T)} w(\J_1(\bar S, S_{i-1} \cup \{x\})).
$$
\end{obs}
\begin{proof}
Indeed, the observation says that a sequence $\bar S$ can be extended only by the elements of the set $Y_i(T)$ and, therefore, its weight is bounded by the sum of the weights of all possible extensions.
\end{proof}
For $q \ge 0$ let 
\begin{equation}\label{sup}
    f(q) = \max_{\bar S = (S_0, S_1, \ldots, S_q)} w(\J_1(\bar S)).
\end{equation}
The following proposition will finish the proof. Note that $\tau(\B) \ge r/l$ because any element $x \in X$ covers at most $l$ sets from $\B$.

\begin{prop}\label{recu}
For any $q \in [0, r/l]$ we have
$$
f(q) \le \prod_{i = q}^{[r/l]-1} \left (1 - \frac{i-1}{n l}\right).
$$
\end{prop}

\begin{proof}
The proof is by induction. The base case $q = [r/l]$ states that $f([r/l]) \le 1$ which we already know by the Erd{\H o}s--Lov{\' a}sz argument.

For the induction step, let $T \in \T_1$ be a covering on which the maximum in (\ref{sup}) is attained. Now apply Observation \ref{obb} and the induction hypothesis to conclude that
$$
f(q) \le \frac{1}{n} |Y_i(T)| f(q+1) \le \left(1 - \frac{i-1}{nl} \right)f(q+1),
$$
where $T$ corresponds to a maximizer of the supremum on the left hand side.
\end{proof}

Substituting $q = 0$ in Proposition \ref{recu} we get
$$
w(\J_1) = f(0) \le \prod_{i = 1}^{[r/l]-1} (1 - \frac{i-1}{n l}) \le \left(  1 - \frac{r}{2n l^2}  \right)^{r/2l} \le e^{ - \frac{r^2}{4l^3 n}}.
$$
Let $y = \frac{r^2}{l^3 n}$. By assumption we have $y \le 1$ and so we have the following elementary inequality: $e^{-y/4}+1 \le 2 e^{-y/10}$. By (\ref{cj}), the desired inequality (\ref{lel}) follows.
\end{proof}

The following simple corollary will be more convenient to combine with Lemma \ref{crlm} in the proof of Theorem \ref{el}.

\begin{cor}\label{corbd}
Let $n \ge 1$ and $r \ge 2l$ be such that $r^2 \le l^3 n$. Let $\B$ be an $n$-uniform intersecting family of size $r$ such that every $l$ distinct sets from $\B$ have an empty intersection. Then for $k \le \frac{r}{20 l^3}$ we have
\begin{equation*}
    c_{n - k}(\B) \le 1.
\end{equation*}
\end{cor}

\begin{proof}
Note that any minimal covering of $\B$ has size at most $|\B| = r$. So for any $\lambda \le 1$ we have
$$
c_{\lambda n}(\B) \le \lambda^{-r} c_n(\B).
$$
By Lemma \ref{lbodeg2}, if we let $\lambda = e^{- \frac{r}{10 l^3 n}}$ then $c_{\lambda n}(\B) \le 1$. Now if $k \le \frac{r}{20 l^3}$ then
$$
\frac{n-k}{n} \ge 1 - \frac{r}{20 l^3 n} \ge e^{-\frac{r}{10 l^3 n}},
$$
which implies that $c_{n-k}(\B) \le 1$.
\end{proof}

\section{Proof of Theorem \ref{main}}\label{sec3}

In this section we put all developed machinery together to prove Theorem \ref{main}. We restate the theorem below for convenience.

\begin{theorem}\label{main2}
For all $\varepsilon > 0$ and sufficiently large $n > n_0(\varepsilon)$ we have the following. Let $\A$ be an intersecting $n$-uniform family. Then
\begin{equation}\label{maine2}
c_n(\A) \le e^{1- \frac{\tau(\A)^{1.5-\varepsilon}}{n}}.
\end{equation}
\end{theorem}

Now we begin the proof of Theorem \ref{main}. Fix $n > n_0(\varepsilon)$ and suppose that there exists an intersecting family $\A$ which violates (\ref{maine}). Let $\A$ be any such family of minimal possible size. In particular, $\A$ is a $\tau$-critical family and $\tau(\A) > n^{2/3}$ because otherwise the right hand side of (\ref{maine2}) is greater than 1 and so we done by Corollary \ref{st}. 

By the minimality of $\A$, for any proper subfamily $\A' \subset \A$ we have
\begin{equation}\label{apr}
    c_n(\A') \le e^{1- \frac{\tau(\A')^{1.5-\varepsilon}}{n}}.
\end{equation}
We are going to apply Lemma \ref{crlm} to various subfamilies of $\A$ and $f(t) = t^{1.5-\varepsilon} - 1$. Let $\lambda = e^{-f'(\tau(\A))} = e^{-(1.5-\varepsilon) \frac{\tau(\A)^{0.5-\varepsilon}}{n}}$ and $k =\sqrt{\tau(\A)}$. 

\begin{prop}\label{pgap}
For any $A_1, A_2 \in \A$ we have 
\begin{equation}\label{gap}
    |A_1 \cap A_2| \not \in [k, n-k].
\end{equation}
\end{prop}

\begin{proof}
Suppose that there are some $A_1, A_2 \in \A$ such that $|A_1 \cap A_2| = x \in [k, n-k]$. Denote $\A' = \{A_1, A_2\}$ and note that
\begin{equation}\label{e2}
    c_{n-k/3}(\A') = \frac{x}{n-k/3} + \frac{(n-x)^2}{(n-k/3)^2} \le 1,
\end{equation}
where the latter inequality holds for every $x \in [k, n-k]$ and any $k \le 0.1 n$. 

Since $\frac{n-k/3}{n} \le \lambda$ for sufficiently large $n$, by Lemma \ref{crlm} applied to $\A'$ we deduce that (\ref{apr}) holds for $\A$ as well. This is however a contradiction to our initial assumption that $\A$ does not satisfy (\ref{maine2}).
\end{proof}

Now we define a relation $\sim$ on $\A$ as follows: two sets $A_1$, $A_2 \in \A$ are equivalent if $|A_1 \cap A_2| \ge n/2$. Then Proposition \ref{pgap} implies that $\sim$ is an equivalence relation on $\A$. Let 
\begin{equation}\label{ecd}
    \A = \K_1 \cup \ldots \cup \K_N
\end{equation}
be the equivalence class decomposition on $\A$ corresponding to $\sim$. This means that for every $i = 1, \ldots, N$ and any $F_1, F_2 \in \K_i$ we have $|F_1 \cap F_2| \ge n-k$ and for any $i \neq j$ and $F_1 \in\K_i$ and $F_2 \in \K_j$ we have $|F_1 \cap F_2| \le k$.

\begin{prop}\label{pr2}
For every $i = 1, \ldots, N$ we have $|\bigcap \K_i| \ge n-5k$.
\end{prop}

\begin{proof}
Suppose that $|\bigcap \K_i| < n-5k$ for some $i$. Let $F \in \K_i$ be an arbitrary set from $\K_i$ and let $K \subset F$ be any subset of size $(n-k)$. Lemma \ref{ker2} applied to the family $\K_i$ and the set $K$ implies that $c_{n-k}(\K_i) \le 1$. So Lemma \ref{crlm} applied to $\K_i$ implies that $\A$ satisfies (\ref{apr}), a contradiction.
\end{proof}

\begin{prop}
We have $|\A| \le n^{6k}$.
\end{prop}

\begin{proof}
Indeed, by Lemma \ref{kerup}, Proposition \ref{pr2} and $\tau$-criticality of $\A$ we have $|\K_i| \le {\tau(\A)+5k \choose 5k}$ for any $i = 1, \ldots, N$. 

Now let $A_i \in \K_i$ be arbitrary representatives. Note that $|A_i \cap A_j| \le k$ for any $i \neq j$. Obviously $k \ll \frac{n}{\log n}$, so by Lemma \ref{small2} we either have $N \le n^{C'}$ or there is a proper subfamily $\A' \subset \A$ such that 
$$
2^{\tau(\A)}c_n(\A) \le 2^{\tau(\A')} c_n(\A') \le 2^{\tau(\A')} C \lambda^{\tau(\A')},
$$
which immediately implies (\ref{maine}). This implies that we in fact have $N \le n^{C'}$ and so
$$
|\A| \le n^{C'} {\tau(\A)+5k \choose 5k} \le n^{6k},
$$
provided that $n$ is large enough.
\end{proof}

Denote $m = \log n^{6k} = 6k \log n$ and let $l = 10\varepsilon^{-1}$. By Lemma \ref{dec2}, there is a subfamily $\A' \subset \A$ such that $\tau(\A\setminus \A') \le \tau(\A)/2$ such that for every $i = 2, \ldots, l$ we have
\begin{equation}\label{expec}
    \E_{A_1, \ldots, A_i \in \A'} |A_1 \cap \ldots \cap A_i| \le C_l \left( \frac{m}{\tau(\A)} \right)^{i-1}n,
\end{equation}
for some new constant $C_l \ll 2^{l^2}$.
Let $r = n^{-\varepsilon} \frac{\tau(\A)}{m}$ (note that $r \gg 1$ since $\tau(\A) \ge n^{2/3}$ by assumption).

Sample uniformly and independently sets $B_1, \ldots, B_r \in \A'$ and form a random family $\B = \{B_1, \ldots, B_r\}$. Applying (\ref{expec}) to all $l$-element intersections in $\B$ we get
\begin{equation*}
   \E \sum_{S \in {[r] \choose l}} \left|\bigcap_{i \in S} B_i\right| \le C_l {r \choose l} \left( \frac{m}{\tau(\A)}\right)^{l-1}n \le C_l n^{1- \varepsilon l} \frac{\tau(\A)}{m} \le n^{2-\varepsilon l} < 1
\end{equation*}
for sufficiently large $n$.

So there exists an $r$-element family $\B \subset \A'$ such that all $l$-wise intersections of sets from $\B$ are empty. By Corollary \ref{corbd}, for $h = \frac{r}{20l^3}$ we have $c_{n-h}(\B) \le 1$. But 
$$
\frac{n-h}{n} \le 1 - \frac{\tau(\A)}{20l^3 m n^{1+\varepsilon}} \le 1 -\frac{\tau(\A)}{k n} \le \lambda,
$$
by the choice of $k = \sqrt{\tau(\A)}$ and $\lambda = e^{-(1.5-\varepsilon) \frac{\tau(\A)^{0.5-\varepsilon}}{n}}$ and sufficiently large $n$.
So by Lemma \ref{crlm} applied to $\F' = \B$ we have (\ref{maine2}). Theorem \ref{main} is proved. 

\section{Remarks}\label{sec4}

Let us describe a construction of a maximal intersecting family which generalizes examples from \cite{EL} and \cite{FOT}. Let $G$ be a tournament on the vertex set $\{1, \ldots, m\}$ and let $K_1, \ldots, K_m$ be a sequence of disjoint non-empty sets. Let $\K_i$ be the family of all sets $F$ such that $K_i \subset F$ and for $i \neq j$:
\begin{align*}
    |F \cap K_j| = \begin{cases}
        1,\text{ if $(i, j) \in G$,}\\
        0, \text{ if $(i, j) \not \in G$.}
    \end{cases}
\end{align*}
It is clear from this definition that the family $\F = \K_1 \cup \ldots \cup \K_m$ is intersecting. Let $d_i$ be the outdegree of the vertex $i$ and let $n > \max d_i$. If we let $|K_i| = n-d_i$ then the family $\F$ is $n$-uniform and intersecting. 

It is not difficult to characterize all minimal coverings of $\F$. First, observe that if $T$ is a minimal covering of $\F$ then $|T \cap K_i| \in \{0, 1, |K_i|\}$. Then for every minimal covering $T$ we can define two sets $A, B \subset [m]$, namely, $A$ is the set of all $i$ such that $|C \cap K_i| = 1$ and $B$ is the set of all $i$ such that $K_i \subset T$. Now the fact that $T$ is a covering is equivalent to the assertion that $A \cup B \cup N_{in}(B) = [m]$, where $N_{in}(B)$ denotes the set of all vertices of $G$ from which there is an edge to $B$.

\paragraph{Example 1.} If we let $G$ to be the linearly ordered complete directed graph then $\F$ coincides with the family constructed by Erd{\H o}s--Lov{\' a}sz \cite{EL}. In this case $\tau(\F) = n$ and $|\F|$ is approximately $n!$.

\paragraph{Example 2.} Let $G$ be graph on the vertex set $\Z_{2t-1} \cup \{v\}$, where the vertices $i, j \in \Z_{2t-1}$ from the cyclic group are connected if $j-i \in [1, t-1]$ and the vertex $v$ has outdegree $2t-1$. In this case we have $n = 2t$, $\tau(\F) = n$ and $|\F|$ is approximately $\left(\frac{n}{2} \right)^{n}$. Note that the main contribution to the size of $\F$ comes from the family $\K_v$ corresponding to the vertex $v$.\footnote{The construction of a maximal intersecting family of size $(n/2)^n$ for odd $n$ is a bit more delicate, see \cite{FOT}.}

It is not hard to see that the construction in the second example gives the maximum size of $\F$ among all constructions of this type.

The construction above and the decomposition (\ref{ecd}) which we used in the proof of Theorem \ref{el} suggest to consider the following special class of families. 
Let $K_1, \ldots, K_n$ be disjoint sets such that $|K_i| = n-a_i$. Let $\K_i$ be an $n$-uniform family of sets containing $K_i$. Let $\F = \K_1 \cup \ldots \cup \K_n$ and suppose that $\tau(\F) = n$ and $\F$ is intersecting.

\begin{conj}[\cite{K}]\label{kconj}
In the situation described above we have $\sum_{i = 1}^n a_i \ge {n \choose 2}$. Moreover, the condition that $|K_i| = n-a_i$ can be replaced by the condition that $\K_i$ is $(|K_i|+a_i)$-uniform.
\end{conj}

Note that, if true, Conjecture \ref{kconj} is best possible: we can take $G$ to be a graph whose outdegrees are precisely $a_1, \ldots, a_n$ and then use the construction of an $n$-uniform family $\F$ described above. One can easily produce sequences of degrees $a_1, \ldots, a_n$ such that the corresponding graph exists and $\tau(\F) = n$.

Note that if there is a counterexample to Conjecture 1
such that, say, $a_i \sim n^{1-\varepsilon}$ for some $\varepsilon > 0$ and every $i = 1, \ldots, n$, then one can construct a very large maximal intersecting family as follows. 

Let $\F_0 = \K_1 \cup \ldots \cup \K_n$ be an $n$-uniform family such that $\tau(\F_0) = n$. Then any set $F$ such that $|F \cap K_i| = 1$, for every $i$, is a minimal covering of $\F_0$. Denote the family of all such sets by $\F'_1$. We have
$$
c_n(\F_0) \ge n^{-n}|\F'_1| = n^{-n} \prod_{i = 1}^n (n - a_i) \sim (1 - n^{-\varepsilon})^n \sim e^{- n^{1-\varepsilon}}.
$$
Moreover, if we let $\F_1$ be an $(n+1)$-uniform family of sets $F \cup \{x_0\}$, where $F \in \F'_1$ and $x_0$ is a ``new" element of the ground set, then $\F = \F_0 \cup \F_1$ is an intersecting family of $n$ and $n+1$ element sets such that $\tau(\F) = n$ and each member of $\F$ is a minimal covering of $\F$. The family $\F$ has size at least $e^{-n^{1-\varepsilon}} n^n$ and so it essentially contradicts the conjecture of Frankl--Ota--Tokushige \cite{FOT}.

In the setting of Conjecture \ref{kconj} we were only able to prove the lower bound $\sum_{i=1}^n a_i \gg n^{3/2}$ but any improvement seems to require new ideas.

\paragraph{Acknowledgements.} I thank Andrey Kupavskii for valuable discussions and for telling me about Conjecture \ref{kconj}. I thank P\'eter Frankl, Stijn Cambie and the anonymous referee for useful comments on an earlier version of this paper.

\end{document}